\documentclass[11pt,reqno]{amsart}

\usepackage{amsmath}
\usepackage{amsfonts}
\usepackage{amssymb}
\usepackage{amsthm}
\usepackage{esint}
\usepackage{mathrsfs}
\usepackage{mathtools}
\usepackage{tensor}

\usepackage{ucs}
\usepackage[utf8x]{inputenc}
\usepackage{color}

\usepackage{verbatim}
\usepackage{url}
\usepackage{hyperref}
\usepackage{enumerate}
\usepackage{comment}

\newcommand{\RR}{\mathbf{R}}

\renewcommand{\SS}{\mathbf{S}}

\newcommand{\JJ}{\mathbf{J}}

\newcommand{\cA}{\mathcal{A}}

\newcommand{\cH}{\mathcal{H}}

\newcommand{\cJ}{\mathcal{J}}
\newcommand{\cL}{\mathcal{L}}
\newcommand{\cM}{\mathcal{M}}
\newcommand{\cN}{\mathcal{N}}

\newcommand{\cP}{\mathcal{P}}
\newcommand{\cQ}{\mathcal{Q}}
\newcommand{\cR}{\mathcal{R}}
\newcommand{\cS}{\mathcal{S}}

\DeclareMathOperator{\ind}{ind}
\DeclareMathOperator{\nul}{nul}

\DeclareMathOperator{\support}{spt}

\DeclareMathOperator{\dist}{dist}

\DeclareMathOperator{\intr}{int}

\DeclareMathOperator{\loc}{loc}

\newcommand{\closure}[1]{\overline{#1}}

\newcommand{\indt}[1]{{\bf 1}_{#1}}

\theoremstyle{definition} \newtheorem{defi}{Definition}
\theoremstyle{plain} \newtheorem{rema}[defi]{Remark}
\theoremstyle{definition} \newtheorem{theo}[defi]{Theorem}
\theoremstyle{definition} \newtheorem{prop}[defi]{Proposition}
\theoremstyle{definition} 
\theoremstyle{definition} \newtheorem{lemm}[defi]{Lemma}
\theoremstyle{definition} 
\theoremstyle{plain} \newtheorem*{exam}{Example}
\theoremstyle{definition} 
\theoremstyle{definition} 
\theoremstyle{definition} \newtheorem*{clai}{Claim}

\numberwithin{defi}{section} 
\numberwithin{equation}{section} 

\allowdisplaybreaks

\title[A note on the index of $2k$-ended phase transitions in $\RR^2$]{A note on the Morse index of $2k$-ended phase transitions in $\RR^2$}
\author{Christos Mantoulidis}
\address{Christos Mantoulidis, Department of Mathematics, Building 380, Sloan Hall, Stanford, CA 94305. c.mantoulidis@math.stanford.edu}

\begin{document}

\maketitle

\begin{abstract}
	We show that the Morse index of every $2k$-ended solution of the Allen-Cahn equation in $\RR^2$ is $\geq k-1$. This bound is expected to be sharp.
\end{abstract}

\section{Introduction} \label{sec:ac.introduction}

The Allen-Cahn equation is an elliptic partial differential equation that describes phase separation in multi-component alloy systems. It is given by:
\begin{equation} \label{eq:ac.pde}
	\Delta u = W'(u) \text{,}	
\end{equation}
where $W$ is a double-well energy potential. The one-dimensional case is important; it admits heteroclinic solutions $H : \RR \to \RR$, which are foundational in the theory of phase transitions, and characterized up to translations by
\begin{equation} \label{eq:ac.heteroclinic.solution}
	H' = \sqrt{2 W \circ H}, \; H(0) = 0 \text{.}	
\end{equation}

\begin{defi} \label{defi:ac.double.well.potential}
	A smooth map $W : \RR \to \RR$ is called a double-well potential provided:
	\begin{enumerate}
		\item $W$ is nonnegative, and vanishes at its two global minima $t = \pm 1$:
			\begin{equation} \label{assu:ac.double.well.potential.nonnegative} \tag{H1}
				W \geq 0, \; W(t) = 0 \iff t = \pm 1 \text{;}	
			\end{equation}
		\item $W$ has a unique, nondegenerate, critical point between its global minima, at $t = 0$:
			\begin{equation} \label{assu:ac.double.well.potential.saddle} \tag{H2}
				tW'(t) < 0 \text{ for } 0 < |t| < 1	\text{, and } W''(0) \neq 0 \text{;}
			\end{equation}
		\item $W$ is strictly convex near $\pm 1$:
			\begin{equation} \label{assu:ac.double.well.potential.convex} \tag{H3}
				W''(t) \geq \kappa > 0 \text{ for } |t| > 1-\alpha, \; \alpha \in (0,1) \text{;}
			\end{equation}
	\end{enumerate}
\end{defi}

\begin{exam} \label{exam:ac.canonical.potential}
	The standard double-well potential is $W(t) = \frac{1}{4} (1-t^2)^2$, and the corresponding equation \eqref{eq:ac.pde} is $\Delta u = u^3 - u$.
\end{exam}

\begin{rema}
	There is nothing special about $t = -1$, $0$, $1$; one can replace these points by any other triple points $t_1 < t_2 < t_3$ in Definition \ref{defi:ac.double.well.potential}, and everything will continue to hold just the same, up to relabeling.
\end{rema}

We now turn to the variational nature of \eqref{eq:ac.pde}. The equation in question arises as the Euler-Lagrange equation for the functional
\begin{equation} \label{eq:ac.energy}
	E[u] \triangleq \int \frac{1}{2} \Vert \nabla u \Vert^2 + W(u) \text{,}
\end{equation}
i.e., functions $u \in C^\infty_{\loc}$ that are zeroes of the {\bf first variation} operator $\delta E[u] : C^\infty_c \to \RR$,
\begin{equation} \label{eq:ac.first.variation}
	\delta E[u]\{\zeta\} \triangleq \int \langle \nabla u, \nabla \zeta \rangle + W'(u) \zeta \text{.}	
\end{equation}

One can make precise the notions of {\bf stability} and {\bf Morse index} in the Allen-Cahn setting by turning to the {\bf second variation} operator, $\delta^2 E[u] : W^{1,2}_0 \otimes W^{1,2}_0 \to \RR$,
\begin{equation} \label{eq:ac.second.variation}
	\delta^2 E[u]\{\zeta, \psi\} \triangleq \int \langle \nabla \zeta, \nabla \psi \rangle + W''(u)\zeta \psi \text{.}
\end{equation}
The Morse index of a critical point measures the number of linearly independent unstable directions for energy. From a physical perspective, unstable critical points are a lot less likely to be observed than stable ones.

\begin{defi}[Stability] \label{defi:ac.stability}
	A critical point $u$ of $E$ is called stable inside an open set $U$ if $\delta^2 E[u] \{ \zeta, \zeta \} \geq 0$ for all $\zeta \in W^{1,2}_0(U)$.
\end{defi}

\begin{defi}[Morse index] \label{defi:ac.index}
	A critical point $u$ of $E$ is said to have Morse index $k$ inside an open set $U$, denoted $\ind(u; U) = k$, provided
	\begin{multline*}
		\max \{ \dim V : V \subset W^{1,2}_0(U) \text{ is a vector space such that } \\
		\delta^2 E[u]\{ \zeta, \zeta\} < 0 \text{ for all } \zeta \in V \setminus \{ 0 \} \} \text{.}
	\end{multline*}
	If the choice of the open subset $U$ is clear then we will simply write $\ind(u)$. Stable critical points, by definition, have $\ind(u) = 0$.
\end{defi}

These notions of stability and Morse index are developed in the appendix (Section \ref{sec:app.pde.morse.index.spaces.quadratic.area.growth}); the relevant operator there is the {\bf Jacobi operator}, $-\Delta + W''(u)$.

In this note we prove the following theorem, which relates the Morse index of $2k$-ended solutions of \eqref{eq:ac.pde} (see Section \ref{sec:ac.m2k}) to their structure at infinity.

\begin{theo} \label{theo:ac.m2k.index.lower.bound}
	Let $u \in \cM_{2k}$, $k \geq 2$. Then $\ind(u) \geq k-1$.
\end{theo}

\begin{rema} \label{rema:ac.m2k.index.lower.bound.assumption}
	Notice that we assume $k \geq 2$, as all elements of $\cM_2$ are stable; this follows, for instance, from the work of K. Wang \cite{Wang14}.
\end{rema}

\begin{rema}
	K. Wang and J. Wei have announced \cite{WangWei17} that $2k$-ended solutions of \eqref{eq:ac.pde} in $\RR^2$ are in bijection with the set of all solutions with $\ind(u) < \infty$. In the same paper, they  independently obtained a special case of Theorem \ref{theo:ac.m2k.index.lower.bound}, namely, that Morse index $1$ is only attained by four-ended solutions.
\end{rema}

\thanks{{\bf Acknowledgements.} This work forms a portion of the author's Ph.D. thesis at Stanford University. The author would like to thank R. Schoen, L. Simon, R. Mazzeo, L. Ryzhik, O. Chodosh, and C. Li for stimulating conversations and their interest in this work, as well as the Department of Mathematics at the University of California, Irvine, where part of this research was carried out. This research was partially supported by the Ric Weiland Graduate Fellowship at Stanford University, and NSF grant DMS-1613603.}

\section{Moduli space $\cM_{2k}$ of $2k$-ended solutions} \label{sec:ac.m2k}

Del Pino-Kowalczyk-Pacard defined in \cite{DelPinoKowalczykPacard13} a space $\cM_{2k}$ of solutions of \eqref{eq:ac.pde} on $\RR^2$ that looks near infinity like a collection of $2k$ copies of the one-dimensional heteroclinic solution \eqref{eq:ac.heteroclinic.solution}. We recall the construction of this space here (after \cite{DelPinoKowalczykPacard13}) for the sake of completeness.

Fix $k \in \{1, 2, \ldots \}$. We denote by $\Lambda^{2k}$ (denoted $\Lambda^{2k}_{\operatorname{ord}}$ in \cite{DelPinoKowalczykPacard13}) the space of ordered $2k$-tuples $\lambda = (\lambda_1, \ldots, \lambda_{2k})$ of oriented affine lines on $\RR^2$, parametrized as
\[ \lambda_j = (r_j, \mathbf{f}_j) \in \RR \times \SS^1, \; j = 1, \ldots, 2k \text{,} \]
where $\mathbf{f}_j = (\cos \theta_j, \sin \theta_j)$, $\theta_1 < \ldots < \theta_{2k} < 2\pi + \theta_1$. For $\lambda \in \Lambda^{2k}$, we denote
\begin{equation} \label{eq:ac.m2k.minimum.angle}
	\theta_\lambda \triangleq \frac{1}{2} \min \{ \theta_2 - \theta_1, \ldots, \theta_{2k} - \theta_{2k-1}, 2\pi + \theta_1 - \theta_{2k} \} \text{.}	
\end{equation}
Fix $\lambda \in \Lambda^{2k}$. For large $R > 0$ and all $j = 1, \ldots, 2k$, there exists $s_j \in \RR$ such that $r_j \JJ \mathbf{f}_j + s_j \mathbf{f}_j \in \partial B_R(\mathbf{0})$, the half-lines $\lambda_j^+ \triangleq r_j \JJ \mathbf{f}_j + s_j \mathbf{f}_j + \RR_+ \mathbf{f}_j$ are disjoint and contained in $\RR^2 \setminus B_R(\mathbf{0})$, and the minimum distance of any two distinct $\lambda_i^+$, $\lambda_j^+$ is $\geq 4$. (Here, $\JJ \in \operatorname{End}(\RR^2)$ is the counterclockwise rotation map by $\frac{\pi}{2}$.) The affine half-lines $\lambda_1^+, \ldots, \lambda_{2k}^+$ and the circle $\partial B_R(\mathbf{0})$ induce a decomposition of $\RR^2$ into $2k+1$ open sets,
\begin{multline} \label{eq:ac.m2k.r2.decomposition}
	\Omega_0 \triangleq B_{R+1}(\mathbf{0}) \text{, and} \\
	\Omega_j \triangleq \bigcap_{i \neq j} \{ \mathbf{x} \in \RR^2 \setminus B_{R-1}(\mathbf{0}) : \dist(\mathbf{x}, \lambda_j^+) < \dist(\mathbf{x}, \lambda_i^+) + 2 \}, \\
	j = 1, \ldots, 2k
\end{multline}
Note that these open sets are not disjoint. Then, we define $\chi_{\Omega_0}, \ldots, \chi_{\Omega_{2k}}$ to be a smooth partition of unity of $\RR^2$ subordinate to $\Omega_0, \ldots, \Omega_{2k}$, and such that
\begin{multline} \label{eq:ac.m2k.r2.decomposition.disjoint}
	\chi_{\Omega_0} \equiv 1 \text{ on } \Omega_0' \triangleq B_{R-1}(\mathbf{0}), \text{ and } \\
	\chi_{\Omega_j} \equiv 1 \text{ on } \Omega_j' \triangleq \cap_{i\neq j} \{ \mathbf{x} \in \RR^2 \setminus B_{R+1}(\mathbf{0}) : \dist(\mathbf{x}, \lambda_j^+) < \dist(\mathbf{x}, \lambda_i^+) - 2\}, \\
	j = 1, \ldots, 2k \text{.}
\end{multline}
Note that these new open sets are disjoint. Without loss of generality, $|\chi_{\Omega_j}| + \Vert \nabla \chi_{\Omega_j} \Vert + \Vert \nabla^2 \chi_{\Omega_j} \Vert \leq c_1$ for all $j = 0, \ldots, 2k$, with $c_1 = c_1(\theta_\lambda)$. Finally, we define
\begin{equation} \label{eq:ac.m2k.approximate.solution}
	u_\lambda \triangleq \sum_{j=1}^{2k} (-1)^{j+1} \chi_{\Omega_j} H(\dist^s(\cdot, \lambda_j)) \text{,}	
\end{equation}
where $\dist^s(\cdot, \lambda_j)$ denotes the signed distance to $\lambda_j$, taking $\JJ\mathbf{f}_j$ to be the positive direction. Here, $H$ is the heteroclinic solution \eqref{eq:ac.heteroclinic.solution}.

\begin{defi}[Del Pino-Kowalczyk-Pacard {\cite[Definition 2.2]{DelPinoKowalczykPacard13}}] \label{defi:ac.m2k}
	For $k \geq 1$, we denote
	\begin{equation} \label{eq:ac.m2k.s2k}
		\cS_{2k} \triangleq \bigcup_{\lambda \in \Lambda^{2k}} \{ u \in C^\infty(\RR^2) : u - u_\lambda \in W^{2,2}(\RR^2) \} \text{.}
	\end{equation}
	We endow $\cS_{2k}$ with the weak topology of the operator
	\begin{equation} \label{eq:ac.m2k.J}
		\cJ : \cS_{2k} \to W^{2,2}(\RR^2) \times \Lambda^{2k}, \; \cJ(u) \triangleq (u-u_\lambda, \lambda) \text{.}
	\end{equation}
	Finally, we define the space of ``$2k$-ended solutions'' to be
	\begin{equation} \label{eq:ac.m2k}
		\cM_{2k} \triangleq \{ u \in \cS_{2k} \text{ satisfying } \eqref{eq:ac.pde} \} \text{.}
	\end{equation}
\end{defi}

\begin{rema} \label{rema:ac.m2k.dimension}
	It was shown in \cite[Theorem 2.2]{DelPinoKowalczykPacard13} that $\cM_{2k}$ is a $2k$-dimensional Banach manifold in neighborhoods of points $u \in \cM_{2k}$ satisfying a certain ``nondegeneracy'' condition, namely, the nonexistence of exponentially decaying Jacobi fields (see Definition \ref{defi:ac.jacobi.fields}).
\end{rema}

\begin{exam}[One-dimensional solutions] \label{exam:ac.m2.one.dimensional}
	Elements of $\cM_2$ are the lifts of one-dimensional heteroclinic solutions to $\RR^2$,
	\[ \RR^2 \ni \mathbf{x} \mapsto H(\langle \mathbf{x}, \mathbf{e} \rangle - \beta) \text{, with } (\mathbf{e}, \beta) \in \SS^1 \times \RR \text{.} \]
\end{exam}

The following result, due to Del Pino-Kowalczyk-Pacard, significantly improves the a priori $W^{2,2}$ decay of $u - u_\lambda$ to an exponential decay:

\begin{theo}[Refined asymptotics, Del Pino-Kowalczyk-Pacard {\cite[Theorem 2.1]{DelPinoKowalczykPacard13}}] \label{theo:ac.m2k.refined.asymptotics}
	Let $u_0 \in \cM_{2k}$. There exists a neighborhood $U \subset \cM_{2k}$ and a $\delta = \delta(u_0) > 0$ such that
	\begin{equation} \label{eq:ac.m2k.refined.asymptotics.i}
		\cJ(U) \subseteq e^{-\delta \Vert \mathbf{x} \Vert} W^{2,2}(\RR^2) \times \Lambda^{2k} \text{,}
	\end{equation}
	and, moreover, such that the restricted map
	\begin{equation} \label{eq:ac.m2k.refined.asymptotics.ii}
		\cJ|_U : U \to e^{-\delta \Vert \mathbf{x} \Vert} W^{2,2}(\RR^2) \times \Lambda^{2k}	
	\end{equation}
	is continuous with respect to the corresponding topologies; here, $\cJ$ is the map defined in \eqref{eq:ac.m2k.J}.
\end{theo}

\section{Proof of Theorem \ref{theo:ac.m2k.index.lower.bound}}

To prove this Theorem, we will need to obtain a precise pointwise understanding of kernel elements of the Jacobi operator, seeing as to how they will play a significant role in the relevant variational theory:

\begin{defi}[Jacobi fields] \label{defi:ac.jacobi.fields}
	If $u$ is a critical point of $E$ in $U$, then the space of its Jacobi fields consists of all functions $v$ that satisfy $- \Delta v + W''(u) v = 0$ in $U$ in the classical sense.
\end{defi}

Denote $R$, $\lambda \in \Lambda^{2k}$, and $u_\lambda$ the objects associated with $u$ by its construction as an element of $\cM_{2k}$ in Section \ref{sec:ac.m2k}. Also, denote $\lambda = (\lambda_1, \ldots, \lambda_{2k})$, with $\lambda_i = (\tau_i, \mathbf{f}_i) \in \RR \times \SS^1$. Recall from \cite[(2.16)]{DelPinoKowalczykPacard13} that:
\begin{equation} \label{eq:ac.m2k.index.lower.bound.i}
	\sum_{i=1}^{2k} \mathbf{f}_i = \mathbf{0} \text{,}	
\end{equation}
and that, after possibly enlarging $R > 0$, $\{ u = 0 \} \setminus B_R(\mathbf{0})$ decomposes into $2k$ disjoint curves $\Gamma_i$, $i = 1, \ldots, 2k$, and, for some $\delta < \theta_\lambda(u)$, $\Gamma_i \subset S(\mathbf{f}_i, \delta/2, R)$ with $S(\mathbf{f}_i, \delta, R)$ all pairwise disjoint. Here,
\begin{equation} \label{eq:ac.m2k.index.lower.bound.ii}
	S(\mathbf{e}, \theta, R) \triangleq \{ r \mathbf{f} : r \geq R, \dist_{\SS^1}(\mathbf{f}, \mathbf{e}) < \theta \} \text{.}
\end{equation}
Finally, using Theorem \ref{theo:ac.m2k.refined.asymptotics} we see that, perhaps after shrinking $\delta > 0$ and enlarging $R > 0$, and perhaps after an ambient rigid motion,
\begin{equation} \label{eq:ac.m2k.index.lower.bound.iii}
	\frac{\nabla u}{\Vert \nabla u \Vert} \approx (-1)^{i+1} \JJ \mathbf{f}_i \text{ in } S(\mathbf{f}_i, \delta/2, R) \text{.}
\end{equation}
		
Lay out $\mathbf{f}_1, \ldots, \mathbf{f}_{2k} \in \SS^1$, and color them red (negative) or blue (positive) depending on the sign of $\langle (-1)^{i+1} \JJ \mathbf{f}_i, \mathbf{e} \rangle$. Here, $\mathbf{e} \in \SS^1$ is a fixed direction, chosen generically, so that $\langle \JJ \mathbf{f}_i, \mathbf{e} \rangle \neq 0 \text{ for all } i = 1, \ldots, 2k$. We will temporarily need the following generalization of $\JJ$:
\[ \JJ_\theta \in \operatorname{End}(\RR^2) \text{ acting by } \begin{bmatrix} \cos \theta & -\sin \theta \\ \sin \theta & \cos \theta \end{bmatrix} \text{.} \]
There exist unique $\varphi_1, \ldots, \varphi_{2k} \in (0, 2\pi)$ such that $\mathbf{f}_{i+1} = \JJ_{\varphi_i}(\mathbf{f}_i)$ for all $i = 1, \ldots, 2k$. It's easy to see that
\begin{equation} \label{eq:ac.m2k.index.lower.bound.iv}
	\varphi_i \in (0, \pi) \text{ for all }	 i = 1, \ldots, 2k
\end{equation}
by combining \eqref{eq:ac.m2k.index.lower.bound.i} with $k \geq 2$ (recall Remark \ref{rema:ac.m2k.index.lower.bound.assumption}).

\begin{clai}
	If $\mathbf{f}_{2\ell-1}$, $\mathbf{f}_{2\ell}$ have the same color, blue, then
	\[ \JJ_{-\varphi_{2\ell-1}} \mathbf{e}, \mathbf{f}_{2\ell-1}, \mathbf{e} \text{ lie counterclockwise on } \SS^1 \text{ in the order listed;} \]
	else, if their common color is red, then
	\[ \JJ_{-\varphi_{2\ell-1}} (-\mathbf{e}), \mathbf{f}_{2\ell-1}, -\mathbf{e} \text{ lie counterclockwise  on } \SS^1 \text{ in the order listed.} \] 
\end{clai}
\begin{proof}
	Without loss of generality, let us assume $\ell = 1$. Recall that the respective colors are determined by the signs of $\langle \JJ_{\frac{\pi}{2}} \mathbf{f}_{1}, \mathbf{e} \rangle$ and $\langle -\JJ_{\frac{\pi}{2}} \mathbf{f}_{2}, \mathbf{e} \rangle = \langle \JJ_{-\frac{\pi}{2}+\varphi_{1}} \mathbf{f}_{1}, \mathbf{e} \rangle$.
	
	Denote $\cP \triangleq \{ \mathbf{f} \in \SS^1 : \langle \mathbf{f}, \mathbf{e} \rangle > 0 \}$. 	If both colors are blue, then
	\[ \JJ_{\frac{\pi}{2}} \mathbf{f}_{1} \in \cP \iff \mathbf{f}_{1} \in \JJ_{-\frac{\pi}{2}}(\cP) \]
	and
	\[ \JJ_{-\frac{\pi}{2}+\varphi_{1}} \mathbf{f}_{1} \in \cP \iff \mathbf{f}_{1} \in \JJ_{\frac{\pi}{2} - \varphi_{1}}(\cP) \text{,} \]
	i.e., $\mathbf{f}_{1} \in \JJ_{-\frac{\pi}{2}}(\cP) \cap \JJ_{\frac{\pi}{2} - \varphi_{1}}(\cP)$. Using \eqref{eq:ac.m2k.index.lower.bound.iv}, we see that the three vertices $\JJ_{-\varphi_{1}} \mathbf{e}$, $\mathbf{f}_{1}$, and $\mathbf{e}$, must lie counterclockwise in this order on $\SS^1$.
	
	If both colors are red, then, by a similar argument, $\mathbf{f}_1 \in \JJ_{-\frac{\pi}{2}}(-\cP) \cap \JJ_{\frac{\pi}{2}-\varphi_1}(-\cP)$, and we see that the three vertices $\JJ_{-\varphi_1}(-\mathbf{e})$, $\mathbf{f}_1$, and $-\mathbf{e}$, must lie counterclockwise in this order on $\SS^1$. 
\end{proof}

In a completely analogous manner, one also checks that:

\begin{clai}
	If $\mathbf{f}_{2\ell}$, $\mathbf{f}_{2\ell+1}$ have the same color, blue, then
	\[ \JJ_{-\varphi_{2\ell}} (-\mathbf{e}), \mathbf{f}_{2\ell}, -\mathbf{e} \text{ lie counterclockwise  on } \SS^1 \text{ in the order listed;} \] 
	else, if their common color is red, then
	\[ \JJ_{-\varphi_{2\ell}} \mathbf{e}, \mathbf{f}_{2\ell}, \mathbf{e} \text{ lie counterclockwise on } \SS^1 \text{ in the order listed.} \]
\end{clai}

We now make the following key observation:

\begin{clai}
	There exist at least $2k-2$ groups of consecutive same-colored vertices.
\end{clai}
\begin{proof}[Proof of claim]
	Within the space of valid colorings,
	\begin{equation} \label{eq:ac.m2k.index.lower.bound.v}
		\{\text{existence of blue } \mathbf{f}_{2\ell-1}, \mathbf{f}_{2\ell} \} \cap \{ \text{existence of red } \mathbf{f}_{2m}, \mathbf{f}_{2m+1} \} = \emptyset \text{.}
	\end{equation}
	This follows by combining the previous two claims. Likewise
	\begin{equation} \label{eq:ac.m2k.index.lower.bound.vi}
		\{\text{existence of red } \mathbf{f}_{2\ell-1}, \mathbf{f}_{2\ell} \} \cap \{ \text{existence of blue } \mathbf{f}_{2m}, \mathbf{f}_{2m+1} \} = \emptyset \text{.}
	\end{equation}
	There are now the following cases to consider:
	\begin{enumerate}
		\item There exist three consecutive same-colored vertices. Then, by combining the previous two claims and engaging in elementary angle-chasing, it follows that there do not exist any more consecutive same-colored vertices. In this case, it follows that there are precisely $2k-2$ groups of consecutive same-colored vertices.
		\item There are no three consecutive same-colored vertices. Then, together with \eqref{eq:ac.m2k.index.lower.bound.v}, \eqref{eq:ac.m2k.index.lower.bound.vi}, it follows that there are at least $2k-2$ groups of consecutive same-colored vertices.
	\end{enumerate}
	This concludes the proof of the claim.
\end{proof}

Given this claim, differentiate \eqref{eq:ac.pde} in the direction of $\mathbf{e} \in \SS^1$. We see that $v \triangleq \langle \nabla u, \mathbf{e} \rangle$ satisfies
\begin{equation} \label{eq:ac.m2k.index.lower.bound.vii}
	\Delta v = W''(u) v \text{ in } \RR^2 \text{.}	
\end{equation}

Define
\begin{multline*}
	\cN \triangleq \{ v = 0 \} \text{ (the ``nodal set''), and } \\
	\cS \triangleq \cN \cap \{ \nabla v = \mathbf{0} \} \text{ (the ``singular set'').}
\end{multline*}
By the implicit function theorem, $\cN \setminus \cS$ consists of smooth, injectively immersed curves in $\RR^2$. By \cite{BersJohnSchechter79}, $\cS$ consists of at most countably many points and, for each $p \in \cS$, there exists $r = r(p)$ such that, up to a diffeomorphism of $B_r(p)$,
\begin{multline} \label{eq:ac.m2k.index.lower.bound.viii}
	\cN \cap B_r(p) \approx \text{the zero set of a} \\
	\text{homogeneous even-degree harmonic polynomial.}
\end{multline}

Denote $\Omega_1, \ldots, \Omega_q \subset \RR^2 \setminus \cN$ the nodal domains (i.e., connected components of $\{ v \neq 0 \}$), labeling so that $\Omega_1, \ldots, \Omega_p$ are the unbounded ones, and $\Omega_{p+1}, \ldots, \Omega_q$ are the bounded ones. By virtue of our precise understanding of $\cN$, $\cS$, as discussed above, we know that they are all open, connected, Lipschitz domains.

\begin{rema} \label{rema:ac.m2k.index.lower.bound.q.finite}
	The notation used here implicitly asserts that there are finitely many nodal domains. This follows, a posteriori, by the proof of the following claim and \cite[Theorem 2.8]{KowalczykLiuPacard12}.
\end{rema}

\begin{clai}
	$\ind(u) \geq q-1$.
\end{clai}

\begin{proof}[Proof of claim]
	First, it's standard that for every bounded nodal domain $\Omega_{p+1}, \ldots, \Omega_q$ we have
	\begin{equation} \label{eq:ac.m2k.index.lower.bound.ix}
		\nul(u; \Omega_i) \geq 1 \text{ for all } i = p+1, \ldots, q \text{.}
	\end{equation}
	Now we move on to unbounded nodal domains. It is not hard to see that we have at least two such. Suppose that $\Omega_1$ is an unbounded nodal domain, and suppose $\Omega_2$ is its counterclockwise neighboring unbounded nodal domain. By \eqref{eq:ac.m2k.index.lower.bound.viii}, $v$ attains opposite signs on $\Omega_1$, $\Omega_2$. Thus, $v$ is a bounded, sign-changing Jacobi field in $\Omega_{12} \triangleq \intr{\closure{\Omega}_1 \cup \closure{\Omega}_2}$, which is itself an open, connected, unbounded Lischitz domain. By Lemma \ref{lemm:app.pde.unstable.past.nodal.domain}, $\ind(u; \Omega_{12}) \geq 1$. Since $\Omega_{12}$ is unbounded, we have
	\[ \ind(u; \widetilde{\Omega}_2) \geq 1 \text{ for some bounded } \widetilde{\Omega}_2 \subsetneq \Omega_{12} \]	
	which is itself open, connected, and Lipschitz. Denote $\widetilde{\Omega}_1 = \emptyset$.
	
	Proceeding similarly (and labeling accordingly) in the counterclockwise direction, we can construct disjoint, bounded, open, connected, Lipschitz $\widetilde{\Omega}_3, \ldots, \widetilde{\Omega}_p$, 
	\begin{equation} \label{eq:ac.m2k.index.lower.bound.x}
		\ind(u; \widetilde{\Omega}_i) \geq 1 \text{ for all } i = 2, \ldots, p \text{,}
	\end{equation}
	where $\widetilde{\Omega}_i \subset \intr{(\closure{\Omega}_{i-1} \cup \closure{\Omega}_i)}$. More precisely, at each stage $i$ we have to sacrifice a bounded portion of $\closure{\Omega}_1 \cup \cdots \cup \closure{\Omega}_{i-1} \setminus (\widetilde{\Omega}_1 \cup \cdots \cup \widetilde{\Omega}_{i-1})$ to give rise to a negative eigenvalue on a slight enlargement of $\Omega_i$, which is bounded and disjoint from $\widetilde{\Omega}_1 \cup \cdots \cup \widetilde{\Omega}_{i-1}$.
	
	The claim follows by combining \eqref{eq:ac.m2k.index.lower.bound.ix}, \eqref{eq:ac.m2k.index.lower.bound.x}, and Theorem \ref{theo:app.pde.courant.nodal.domain}.
\end{proof}

We now estimate $q-1$ from below. It will be convenient to assume that $\cS$ and the set of connected components of $\cN \setminus \cS$ are both finite sets---refer to Remark \ref{rema:ac.m2k.index.lower.bound.infinities} for the minor necessary adjustments to deal with the general case. From Euler's formula for planar graphs, we know that
\begin{equation} \label{eq:ac.m2k.index.lower.bound.xi}
	q = 1 + |\{ \text{connected components of } \cN \setminus \cS \}| - |\cS| \text{,}
\end{equation}
where $|\cdot|$ denotes the cardinality of a set. By \eqref{eq:ac.m2k.index.lower.bound.viii}, every connected component $\Gamma$ of $\cN \setminus \cS$ is a smooth curve with
\begin{equation} \label{eq:ac.m2k.index.lower.bound.xii}
	|\partial \Gamma| = 0, 1, \text{ or } 2 \text{,}
\end{equation}
depending on whether $\Gamma$ is infinite in both directions, one direction, or is finite. Counting the set of pairs $(v, e)$ of vertices and edges in $\cN$ in two ways, we see that

\begin{clai}
	$q \geq k$.
\end{clai}
\begin{proof}
	The fact that there exist at least $2k-2$ groups of consecutive same-colored vertices implies that there exists $R > 0$ sufficiently large so that $\cS \subset B_R(\mathbf{0})$ and $\cN \setminus B_R(\mathbf{0})$ has at least $2k-2$ components. By a straightforward counting argument combined with \eqref{eq:ac.m2k.index.lower.bound.xii}, this implies
	\begin{equation} \label{eq:ac.m2k.index.lower.bound.xiii}
		2k-2 \leq \sum_{\ell=0}^2 (2-\ell) \cdot |\{ \text{connected components } \Gamma \subset \cN \setminus \cS : |\partial \Gamma| = \ell \}| \text{.}
	\end{equation}
	On the other hand, by counting the elements of the set
	\[ \cA \triangleq \{ (p, \Gamma) : p \in \cS, \; \Gamma =  \text{connected component of } \cN \setminus \cS \text{ incident to } p \} \]
	in one way, we find that
	\begin{equation} \label{eq:ac.m2k.index.lower.bound.xiv}
		|\cA| = \sum_{\ell=0}^2 \ell \cdot |\{ \text{connected components } \Gamma \subset \cN \setminus \cS : |\partial \Gamma| = \ell \}| \text{.}
	\end{equation}
	Adding \eqref{eq:ac.m2k.index.lower.bound.xiii}, \eqref{eq:ac.m2k.index.lower.bound.xiv}, and rearranging, we get
	\begin{multline} \label{eq:ac.m2k.index.lower.bound.xv}
		2 \cdot |\{ \text{connected components of } \cN \setminus \cS \}| \geq |\cA| + 2k-2 \\
		\iff |\{ \text{connected components of } \cN \setminus \cS \}| \geq \frac{1}{2} |\cA| + k-1 \text{.}
	\end{multline}
	Plugging \eqref{eq:ac.m2k.index.lower.bound.xv} into \eqref{eq:ac.m2k.index.lower.bound.xi} yields the estimate
	\begin{equation} \label{eq:ac.m2k.index.lower.bound.xvi}
		q \geq k + \frac{1}{2} |\cA| - |\cS| \text{.}
	\end{equation}
	On the other hand, because of \eqref{eq:ac.m2k.index.lower.bound.viii}, each $p \in \cS$ contributes at least two elements to $\cA$; i.e., $|\cA| \geq 2 \cdot |\cS|$. The claim follows.
\end{proof}

\begin{rema} \label{rema:ac.m2k.index.lower.bound.infinities}
	The proof above assumed that
	\[ |\cS| + |\{ \text{connected components of } \cN \setminus \cS \}| < \infty \text{,} \]
	so let us discuss the necessary adjustments for it to go through in the general case. By the finiteness of $q$ (see Remark \ref{rema:ac.m2k.index.lower.bound.q.finite}), we know that there exists a large enough radius $R$ so that $\Omega_i \cap B_R(\mathbf{0})$ is connected and nonempty for every $i = 1, \ldots, q$. By the local finiteness of $\cS$, we may further arrange for $\partial B_R(\mathbf{0}) \cap \cS = \emptyset$ and for all intersections $\partial \Omega_i \cap \partial B_R(\mathbf{0})$, $i = 1, \ldots, 2k$, to be transverse. The finite planar graph arrangement contained within $B_R(\mathbf{0})$ has the same number of faces as the original infinite planar graph arrangement. We may, therefore, repeat the previous proof, starting at Remark \ref{rema:ac.m2k.index.lower.bound.q.finite}, discarding all elements of $\cS$ and components of $\cN \setminus \cS$ that lie fully outside of $B_R(\mathbf{0})$, and identifying $\partial B_R(\mathbf{0})$ with infinity.
\end{rema}

Combining everything above, we obtain the thesis of Theorem \ref{theo:ac.m2k.index.lower.bound}.

\appendix

\section{Morse index in unbounded regions on $\RR^2$} \label{sec:app.pde.morse.index.spaces.quadratic.area.growth}

In this section we will study Schr\"odinger operators
\begin{equation} \label{eq:app.pde.schroedinger}
	L \triangleq -\Delta + V, \text{ where } V \in C^\infty_{\loc}(\RR^2) \cap L^\infty(\RR^2)
\end{equation}
on $\RR^2$. Our goal is to extend certain known results on eigenvalues of Schr\"odinger operators on compact domains to a noncompact setting that arises in context of this paper. 

The following lemma is classical:

\begin{lemm}[Logarithmic cutoff functions] \label{lemm:app.pde.log.cutoff.functions}
	For every $R > 0$, there exists $\xi_R \in W^{1,\infty}(\RR^2) \cap C^0_c(\RR^2)$ such that
	\begin{enumerate}
		\item $0 \leq \xi_R \leq 1$, $\xi_R \equiv 1$ on $B_R$, $\xi_R \equiv 0$ outside $B_{R^2}$,
		\item $\lim_{R \uparrow \infty} \Vert \nabla \xi_R \Vert_{L^2} = 0$, and
		\item $\lim_{R \uparrow \infty} \Vert \nabla \xi_R \Vert_{L^\infty} = 0$.
	\end{enumerate}
\end{lemm}
\begin{proof}
	We have already prescribed the behavior of $\xi_R$ of $B_R$ and $\RR^2 \setminus B_{R^2}$, so it remains to define it on $B_{R^2} \setminus B_R$. We do so as
	\[ \xi_R(x) \triangleq 2 - \frac{\log r}{\log R}, \]
	where we write $r = r(x)$ for $\dist(x, 0)$. Such a function satisfies $\xi_R \in W^{1,\infty}(\RR^2) \cap C^0_c(\RR^2)$, and
	\[ \nabla \xi_R = - \frac{\nabla r}{r \log R} \text{ a.e. on } \RR^2 \text{,} \]
	and, therefore, that $\lim_{R\uparrow\infty} \Vert \nabla \xi_R \Vert_{\infty} = 0$. From the coarea formula and integration by parts, one sees that
	\begin{align*}
		\int_{\RR^2} \Vert \nabla \xi_R \Vert^2 \, d\cL^2 & = \frac{1}{\log^2 R} \int_{B_{R^2} \setminus B_R} \frac{d\cL^2}{r^2 \log^2 R} \\
			& = \frac{1}{\log^2 R} \int_R^{R^2} \frac{1}{r^2} \cH^1(\partial B_r) \, dr \\
			& = \frac{1}{\log^2 R} \left[ \frac{1}{r^2} \cL^2(B_r) \right]_{r=R}^{R^2} + \frac{2}{\log^2 R} \int_R^{R^2} \frac{\cL^2(B_r)}{r^3} \, dr \text{.}
	\end{align*}
	This evidently decays as $R \uparrow \infty$.
\end{proof}

We recall that the quadratic form associated to $L$ is $\cQ : W^{1,2}(\RR^2) \otimes W^{1,2}(\RR^2) \to \RR$, with
\begin{equation} \label{eq:app.pde.bilinear.form}
	\cQ(\zeta, \psi) \triangleq \int_{\RR^2} \left[ \langle \nabla \zeta, \nabla \psi \rangle + V \zeta \psi \right] \, d\cL^2 \text{, } \zeta, \psi \in W^{1,2}(\RR^2) \text{,}
\end{equation}
and the corresponding Rayleigh quotient is $\cQ : W^{1,2}(\RR^2) \setminus \{ 0 \} \to \RR$, with
\begin{equation} \label{eq:app.pde.rayleigh.quotient}
	\cR[\zeta] \triangleq \frac{\cQ(\zeta, \zeta)}{\Vert \zeta \Vert_{L^2(\RR^2)}^2} \text{, } \zeta \in W^{1,2}(\RR^2) \setminus \{ 0 \} \text{.}
\end{equation}

\begin{defi}[Morse index, nullity]
	Let $L$ be as in \eqref{eq:app.pde.schroedinger}, and suppose $\Omega \subseteq \RR^2$ is an open, connected, Lipschitz domain. We define the Morse index of $L$ on $\Omega$ as
	\begin{multline} \label{eq:app.pde.index}
		\ind(L; \Omega) \triangleq \sup \Big\{ \dim V : V \subset W^{1,2}_0(\Omega) \text{ a subspace such that} \\
			\cQ(\zeta, \zeta) < 0 \text{ for all } \zeta \in W^{1,2}_0(\Omega) \setminus \{ 0 \} \Big\}
	\end{multline}
	and the ($L^2$-) nullity of $L$ on $\Omega$ as
	\begin{equation} \label{eq:app.pde.nullity}
		\nul(L; \Omega) \triangleq \dim \{ u \in W^{1,2}_0(\Omega) : Lu = 0 \text{ weakly in } \Omega \} \text{.}
	\end{equation}
\end{defi}

The Morse index counts the dimensionality of the space instabilities for a particular critical point. Heuristically, this corresponds to the number of negative eigenvalues, counted with multiplicity.

We note two results for negative eigenvalues. First, they cannot be ``too'' negative. Second, when the Morse index is finite, the space of instabilities can be ``exhausted'' by finitely many eigenfunctions, with negative eigenvalue, in a suitable Sobolev space.

\begin{lemm}  \label{lemm:app.pde.negative.eigenvalue.bound}
	Let $L$ be as in \eqref{eq:app.pde.schroedinger}, $\Omega \subseteq \RR^2$ be an open, connected, Lipschitz domain. For every $f \in W^{1,2}_0(\Omega)$, with $Lf = \lambda f$ weakly in $\Omega$, we have
	\[ -\lambda \leq \max\{ 0, - \inf V \} \text{.} \]
\end{lemm}
\begin{proof}
	Assume, without loss of generality, that $\Vert f \Vert_{L^2(\Omega)} = 1$. Consider cutoff functions $\chi_R \in W^{1,\infty}(\RR^2)$ so that
	\[ \chi_R = 1 \text{ on } B_R, \; \chi_R = 0 \text{ outside } B_{2R}, \; \Vert \nabla \chi_R \Vert \leq c_0 R^{-1} \text{.} \]
	Notice that $\chi_R^2 f \in W^{1,2}_0(\Omega)$, so we may use it as a test function on $Lf = \lambda f$. We get:
	\begin{align*}
		& \int_\Omega \left[ \langle \nabla (\chi_R^2 f), \nabla f \rangle + V\chi_R^2 f^2 \right] \, d\cL^2 = \lambda \int_\Omega \chi_R^2 f^2 \, d\cL^2 \\
		& \qquad \iff \int_\Omega \left[ 2 \chi_R f \langle \nabla \chi_R, \nabla f \rangle + \chi_R^2 \Vert \nabla f \Vert^2 + V \chi_R^2 f^2 \right] \, d\cL^2 \\
		& \qquad \qquad \qquad = \lambda \int_\Omega \chi_R^2 f^2 \, d\cL^2 \text{.}
	\end{align*}
	Using Cauchy-Schwarz on $2 \chi_R f \langle \nabla \chi_R, \nabla f \rangle$,
	\[ -\lambda \int_\Omega \chi_R^2 f^2 \, d\cL^2 \leq \int_\Omega \left[ f^2 \Vert \nabla \chi_R \Vert^2 + \max\{ 0, - \inf V \} \chi_R^2 f^2 \right] \, d\cL^2 \text{.} \]
	The result follows by letting $R \uparrow \infty$.
\end{proof}

The following result should be well-known to experts of elliptic partial differential equations:

\begin{prop}[Negative eigenfunction representation, cf. Fischer-Colbrie {\cite[Proposition 2]{FischerColbrie85}}] \label{prop:app.pde.negative.eigenfunction.representation}
	Let $L$ be as in \eqref{eq:app.pde.schroedinger}, and $\Omega \subseteq \RR^2$ be an open, connected, Lipschitz domain with $k = \ind(L; \Omega) < \infty$. There are $\lambda_1 \leq \ldots \leq \lambda_k < 0$ and $L^2$-orthonormal functions $\varphi_1, \ldots, \varphi_k \in W^{1,2}_0(\Omega) \cap C^\infty_{\loc}(\Omega)$, such that $L \varphi_i = \lambda_i \varphi_i$ weakly in $\Omega$. Moreover,
		\begin{equation} \label{eq:app.pde.stability.inequality}
			\cQ(\zeta, \zeta) \geq 0 \text{ for all } \zeta \in W^{1,2}_0(\Omega) \cap \{ \varphi_1, \ldots, \varphi_k \}^\perp \text{,}
		\end{equation}
		with equality if and only if $\zeta \in \nul(L; \Omega)$; here, $\perp$ is taken with respect to  $L^2(\Omega)$.
\end{prop}
\begin{proof}
	By definition, if $\ind(L; \Omega) = k < \infty$, then there exists $R_0 > 1$ large enough such that $\ind(L; \Omega \cap B_{R}) = k$ for all $R \geq R_0$. Seeing as to how $\Omega \cap B_R$ is precompact for all $R \geq R_0$, and Lipschitz for a.e. $R \geq R_0$, for a.e. $R \geq R_0$ there exists a sequence of eigenvalues $\lambda_{1,R} < \lambda_{2,R} \leq \ldots \leq \lambda_{k,R} < 0$ 	and corresponding eigenfunctions $\varphi_{1,R}, \varphi_{2,R}, \ldots, \varphi_{k,R} \in W^{1,2}_0(\Omega \cap B_R)$ that are $L^2$-orthonormal. By the monotonicity of Dirichlet eigenvalues, we know that $R \mapsto \lambda_{i,R}$ is decreasing in $R$, for each $i = 1, \ldots, k$. In particular, there exists $\mu > 0$ such that
	\begin{equation} \label{eq:app.pde.negative.eigenfunction.representation.i}
		\lambda_{i,R} \leq -\mu \text{ for all } i = 1, \ldots, k \text{ and a.e. } R \geq R_0 \text{.}
	\end{equation}
	Moreover, from Lemma \ref{lemm:app.pde.negative.eigenvalue.bound}, there exists $M = M(V) \geq 0$ such that
	\begin{equation}
		\lambda_{i,R} \geq -M \text{ for all } i = 1, \ldots, k \text{ and a.e. } R \geq R_0 \text{.}
	\end{equation}
	
	\begin{clai}
		There exist $\lambda_1, \ldots, \lambda_k$ such that $-M \leq \lambda_1 \leq \ldots \leq \lambda_k \leq -\mu$, a corresponding orthonormal sequence $\{ \varphi_i \}_{i=1,\ldots,k} \subset W^{1,2}_0(\Omega)$, and a subsequence $R_\ell \uparrow \infty$ such that 
		\[ \lim_{\ell \to \infty} \lambda_{i,R_\ell} = \lambda_i \text{ and } \lim_{\ell \to \infty} \varphi_{i,R_\ell} = \varphi_i \]
		strongly in $L^2(\Omega)$ and weakly in $W^{1,2}(\Omega)$, with $L\varphi_i = \lambda_i \varphi_i$ weakly in $\Omega$.
	\end{clai}
	\begin{proof}[Proof of claim]
		First, note that $L\varphi_i = \lambda_i \varphi_i$ holding true weakly is a direct consequence of all the previous claims, so it suffices to prove those.
		
		In what follows, suppose that $R_0 < S^{1/2} < R^{1/4}$. We will make use of the logarithmic cutoff functions from Lemma \ref{lemm:app.pde.log.cutoff.functions}. Testing $L\varphi_{i,R} = \lambda_{i,R} \varphi$ in $\Omega \cap B_R$ with $\xi_S^2 \varphi_{i,R}$, we find
		\begin{multline*} 
			\int_{\Omega \cap B_R} \Big[ \xi_S^2(\Vert \nabla \varphi_{i,R} \Vert^2 + V \varphi_{i,R}^2 - \lambda_{i,R}\varphi_{i,R}^2) \\
			- 2 \xi_S \varphi_{i,R} \langle \nabla \xi_S, \nabla \varphi_{i,R} \rangle \Big] \, d\cL^2 = 0 \text{.}
		\end{multline*}
		Using Cauchy-Schwarz on the last term, $\lambda_{i,R} \leq 0$, $V \in L^\infty$, and conclusions (1) and (3) of Lemma \ref{lemm:app.pde.log.cutoff.functions}, we get the estimate
		\begin{multline} \label{eq:app.pde.negative.eigenfunction.representation.ii}
			\int_{\Omega \cap B_S} \Vert \nabla \varphi_{i,R} \Vert^2 \, d\cL^2 \\
			\leq \int_{\Omega \cap B_R} \xi_S^2 \Vert \nabla \varphi_{i,R} \Vert^2 \, d\cL^2 \leq c_0 \int_{\Omega \cap B_R} \Vert \varphi_{i,R} \Vert^2 \, \cL^2 = c_0 \text{,}
		\end{multline}
		for a fixed $c_0 > 0$ that applies for all $R_0 < S^{1/2} < R^{1/4}$, $i = 1, \ldots, k$.
		
		Next, recall that $(1-\xi_{S^{1/2}}) \varphi_{i,R} \in W^{1,2}_0(\Omega \cap B_R \setminus \closure{B}_S)$ and that, by our choice of $R_0$, $\ind(L; \Omega \setminus \closure{B}_S) = 0$. From the minmax characterization of eigenvalues on compact domains, and integration by parts, it follows that
		\begin{align*}
			0 & \leq \int_{\Omega \cap B_R} \left[ \Vert \nabla ((1-\xi_{S^{1/2}}) \varphi_{i,R}) \Vert^2 + V (1-\xi_{S^{1/2}})^2 \varphi_{i,R}^2 \right] \, d\cL^2 \\
				& = \int_{\Omega \cap B_R} (1-\xi_{S^{1/2}})^2 \varphi_{i,R} L\varphi_{i,R} \, d\cL^2 + \int_{\Omega \cap B_R} \Vert \nabla (1-\xi_{S^{1/2}}) \Vert^2 \varphi_{i,R}^2 \, d\cL^2 \text{.}
		\end{align*}
		Recalling that $L\varphi_{i,R} = \lambda_{i,R} \varphi_{i,R}$ classically in $\Omega \cap B_R$,  $\lambda_{i,R} \leq -\mu < 0$, and conclusion (3) of Lemma \ref{lemm:app.pde.log.cutoff.functions}, we obtain, after rearranging,
		\begin{align} 
			\int_{\Omega \cap B_R \setminus \closure{B}_S} \varphi_{i,R}^2 \, d\cL^2 & \leq \int_{\Omega \cap B_R} (1-\xi_{S^{1/2}})^2 \varphi_{i,R}^2 \, d\cL^2 \nonumber \\
				& \leq \mu^{-1} \int_{\Omega \cap B_R} \Vert \nabla (1-\xi_{S^{1/2}}) \Vert^2 \varphi_{i,R}^2 \, d\cL^2 \nonumber \\
				& \leq \mu^{-1} \Vert \nabla \xi_{S^{1/2}} \Vert_{L^\infty(\RR^2)}^2 \int_{\Omega \cap B_R} \varphi_{i,R}^2 \, d\cL^2 \nonumber \\
				& = \mu^{-1} \Vert \nabla \xi_{S^{1/2}} \Vert_{L^\infty(\RR^2)}^2 = c_1(\mu, S) \text{,} \label{eq:app.pde.negative.eigenfunction.representation.iii}
		\end{align}
		where $\lim_{S \uparrow \infty} c_1(\mu, S) = 0$ for all $\mu > 0$, by conclusion (3) of Lemma \ref{lemm:app.pde.log.cutoff.functions}. We can now extract a convergent subsequence by recalling that the embedding $W^{1,2}_0(\Omega \cap B_R) \hookrightarrow L^2(\Omega \cap B_R)$ is compact, \eqref{eq:app.pde.negative.eigenfunction.representation.ii}, and the uniformly decaying exterior  $L^2$ bound \eqref{eq:app.pde.negative.eigenfunction.representation.iii}.
	\end{proof}
	
	Suppose, now, that $\zeta \in C^\infty_c(\Omega)$. 	Fix $n_0$ sufficiently large so that $\support \zeta \subset \subset B_{R_\ell}$ for all $\ell \geq n_0$. Let
	\[ \zeta_\perp \triangleq \zeta - \sum_{i=1}^k \langle \zeta, \varphi_i \rangle_{L^2(\Omega)} \varphi_i \in W^{1,2}_0(\Omega) \cap \{ \varphi_1, \ldots, \varphi_k \rangle^\perp \text{,} \]
	and for $\ell \geq n_0$,
	\begin{multline*}
		\zeta_{\ell,\perp} \triangleq \zeta - \sum_{i=1}^k \langle \zeta, \varphi_{i,R_\ell} \rangle_{L^2(\Omega \cap B_{R_\ell})} \varphi_{i,R_\ell} \\
		\in W^{1,2}_0(\Omega \cap B_{R_\ell}) \cap \{ \varphi_{1,R_\ell}, \ldots, \varphi_{k,R_\ell} \}^\perp \text{,}
	\end{multline*}
	where $\perp$ is taken with respect to $L^2(\Omega)$ and $L^2(\Omega \cap B_{R_\ell})$, respectively. Then, by the strong $L^2$ convergence, and the fact that $\cQ(\zeta_{\ell,\perp}, \zeta_{\ell,\perp}) \geq 0$ from the minmax characterization of eigenvalues in the compact setting and $\ind(L; \Omega \cap B_{R_\ell}) = k$, we have
	\begin{align*} 
		\cQ(\zeta_\perp, \zeta_\perp) & = \cQ(\zeta, \zeta) - \sum_{i=1}^k \lambda_i \langle \zeta, \varphi_i \rangle^2_{L^2(\Omega)} \\
			& = \lim_{\ell \to \infty} \left( \cQ(\zeta, \zeta) - \sum_{i=1}^k \lambda_{i,R_\ell} \langle \zeta, \varphi_{i,R_\ell} \rangle_{L^2(\Omega \cap B_{R_\ell})}^2 \right) \\
			& = \lim_{\ell \to \infty} \cQ(\zeta_{\ell,\perp}, \zeta_{\ell,\perp}) \geq 0 \text{.}
	\end{align*}
	Inequality \eqref{eq:app.pde.stability.inequality} follows from this, since $C^\infty_c(\Omega)$ is dense in $W^{1,2}_0(\Omega)$.
	
	For the rigidity case, we proceed as follows. A function $\zeta \in W^{1,2}_0(\Omega) \cap \{ \varphi_1, \ldots, \varphi_k \}^\perp$ attaining equality in \eqref{eq:app.pde.stability.inequality} will be a global minimizer for
	\[ W^{1,2}_0(\Omega) \ni \psi \mapsto \cQ(\psi_\perp, \psi_\perp) = \cQ(\psi, \psi) - \sum_{i=1}^k \lambda_i \langle \psi, \varphi_i \rangle_{L^2(\Omega)}^2 \text{,} \]
	so it will also be a critical point, i.e.,
	\[ \cQ(\zeta, \psi) - \sum_{i=1}^k \lambda_i \langle \zeta, \varphi_i \rangle_{L^2(\Omega)} \langle \psi, \varphi_i \rangle_{L^2(\Omega)} = 0 \text{ for all } \psi \in W^{1,2}_0(\Omega) \text{;} \]
	however, $\zeta \in W^{1,2}_0(\Omega) \cap \{ \varphi_1, \ldots, \varphi_k \}^\perp$, so $Q(\zeta, \psi) = 0$ for all $\psi$, and the claim follows.
\end{proof}

From this we get:

\begin{theo}[Noncompact Courant nodal domain theorem, cf. Montiel-Ros {\cite[Lemma 12]{MontielRos91}}]\label{theo:app.pde.courant.nodal.domain}
	Suppose the open, connected, Lipschitz domain $\Omega$ can be partitioned into open, connected, disjoint, Lipschitz domains $\Omega_1, \ldots, \Omega_m$. For any $L$ as in \eqref{eq:app.pde.schroedinger},
	\begin{equation} \label{eq:app.pde.courant.nodal.domain.theorem} 
		\ind(L; \Omega) \geq \sum_{i=1}^{m-1} \left( \ind(L; \Omega_i) + \nul(L; \Omega_i) \right) + \ind(L; \Omega_m) \text{.}
	\end{equation}
\end{theo}
\begin{proof}
	Without loss of generality, we may assume $\ind(L; \Omega) < \infty$. Invoking Proposition \ref{prop:app.pde.negative.eigenfunction.representation} $m+1$ times, we obtain:
	\begin{enumerate}
		\item $N \subset W^{1,2}_0(\Omega)$ such that $\dim N = \ind(L; \Omega)$,
		\item $N_i \subset W^{1,2}_0(\Omega_i)$ such that $\dim N_i = \ind(L; \Omega)$, $i = 1, \ldots, m$.
	\end{enumerate}
	Likewise, let $K_i = \nul(L; \Omega_i)$, $i = 1, \ldots, m-1$.
	
	Suppose, for the sake of contradiction, that \eqref{eq:app.pde.courant.nodal.domain.theorem} were false. By an elementary dimension counting argument, there would exist a choice of $f_i$, $i = 1, \ldots, m$, where $f_i \in N_i \oplus K_i$, $i = 1, \ldots, m-1$, and $f_m \in N_m$, such that $f = f_1 + \ldots + f_m \in N^\perp \setminus \{ 0 \}$, where $\perp$ is taken with respect to $L^2(\Omega)$, and we have extended each $f_i$ by zero on $\Omega \setminus \Omega_i$. By Proposition \ref{prop:app.pde.negative.eigenfunction.representation},
	\begin{equation} \label{eq:app.pde.courant.nodal.domain.i}
		\cQ(f, f) \geq 0 \text{,}	
	\end{equation}
	with equality if and only if $f \in \nul(L; \Omega)$. On the other hand, since $\support f_i \subseteq \closure{\Omega}_i$, and the domains $\Omega_i$ are Lipschitz and partition $\Omega$,
	\begin{equation} \label{eq:app.pde.courant.nodal.domain.ii}
		\cQ(f, f) = \sum_{i=1}^m \cQ(f_i, f_i) \leq 0 \text{,}
	\end{equation}
	with equality if and only if $f_i \in K_i$, $i = 1, \ldots, m-1$, and $f_m = 0$ a.e. From \eqref{eq:app.pde.courant.nodal.domain.i} and \eqref{eq:app.pde.courant.nodal.domain.ii} we see that equality must hold. In particular, $f \in \nul(L; \Omega)$ and $f_m = 0$ a.e. on $N_m$. In particular, since $f = f_1 + \ldots + f_m$ and $f_1, \ldots, f_{m-1}$ have no essential support in $\Omega_m$, it must be that $f = 0$ a.e. in $\Omega_m$. Since $Lf = 0$ weakly on $\Omega$, unique continuation forces $f = 0$ a.e. in $\Omega$, a contradiction to our having chosen $f \in N^\perp \setminus \{ 0 \}$.
\end{proof}

Finally, the following proposition will be very useful in the noncompact setting. It is motivated by Ghoussoub-Gui's original proof of De Giorgi's conjecture in $\RR^2$ \cite[Theorem 1.1]{GhoussoubGui98}.

\begin{lemm} \label{lemm:app.pde.unstable.past.nodal.domain}
	Let $L$ be as in \eqref{eq:app.pde.schroedinger}, and that $u \not \equiv 0$ is a bounded Jacobi field. If $\Omega \subseteq \RR^2$ is an open, connected, Lipschitz domain, such that $u|_{\partial \Omega} \equiv 0$, then on every open, connected, Lipschitz $\Omega' \supsetneq \Omega$, $\ind(L; \Omega') \geq 1$.
\end{lemm}
\begin{proof}
	We define the auxiliary function $u' = u \indt{\Omega}$. We will argue by contradiction by means of the following claim:
	
	\begin{clai}
		If $\ind(L; \Omega') = 0$, then $Lu' = 0$ weakly on $\Omega'$.
	\end{clai}
		\begin{proof}[Proof of claim]
		Consider cutoff functions $\xi_R$ as in Lemma \ref{lemm:app.pde.log.cutoff.functions}. Multiplying $Lu = 0$ holds classically in all of $\RR^2$ by $\xi_R u$ and integrating by parts over the nodal domain $\Omega$ (recall $u|_{\partial \Omega} \equiv 0$), we get
		\begin{align*}
			0 & = \int_\Omega (-\Delta u + Vu) \xi_R^2 u \, d\cL^2 \\
				& = \int_\Omega \left[ \xi_R^2 \Vert \nabla u \Vert^2 + V \xi_R^2 u^2 + 2 \xi_R u \langle \nabla \xi_R, \nabla u \rangle \right] \, d\cL^2 \\
				& = \int_\Omega \left[ \Vert \xi_R \nabla u + u \nabla \xi_R \Vert^2 + V \xi_R^2 u^2 - \Vert \nabla \xi_R \Vert^2 u^2 \right] \, d\cL^2 \\
				& = \int_\Omega \left[ \Vert \nabla (\xi_R u) \Vert^2 + V \xi_R^2 u^2 - \Vert \nabla \xi_R \Vert^2 u^2 \right] \, d\cL^2 \text{.}
		\end{align*}
		From the definition of $u' = u \indt{\Omega} \in W^{1,2}_0(\Omega')$, we find---by rearranging---that
		\[ \int_{\Omega'} \left[ \Vert \nabla (\xi_R u') \Vert^2 + V \xi_R^2 (u')^2 \right] \, d\cL^2 = \int_{\Omega'} \Vert \nabla \xi_R \Vert^2 (u')^2 \, d\cL^2 \text{.} \]
		Thus, the Rayleigh quotient $\cR$ \eqref{eq:app.pde.rayleigh.quotient} satisfies
		\begin{equation} \label{eq:app.pde.unstable.past.nodal.domain.i}
			\cR[\xi_R u'] = \Vert \xi_R u' \Vert_{L^2(\Omega')}^{-2} \int_{\Omega'} \Vert \nabla \xi_R \Vert^2 (u')^2 \, d\cL^2 \text{.}
		\end{equation}
		From Lemma \ref{lemm:app.pde.log.cutoff.functions} and $u' \in L^\infty$, we find that
		\begin{equation} \label{eq:app.pde.unstable.past.nodal.domain.ii}
			\lim_{R \uparrow \infty} \int_{\Omega'} \Vert \nabla \xi_R \Vert^2 (u')^2 \, d\cL^2 = 0 \text{,}
		\end{equation}
		and, therefore, since $u' \not \equiv 0$ a.e.,
		\begin{equation} \label{eq:app.pde.unstable.past.nodal.domain.iii}
			\lim_{R \uparrow \infty} \cR[\xi_R u'] = 0 	\text{.}
		\end{equation}
		Since we're assuming $\ind(L; \Omega') = 0$, Proposition \ref{prop:app.pde.negative.eigenfunction.representation} and \eqref{eq:app.pde.unstable.past.nodal.domain.iii} together show that $R \mapsto \xi_R u'$ is a minimizing sequence for the Banach space operator $\cR : W^{1,2}_0(\Omega') \to \RR \cup \{ \infty \}$, which is bounded from below by zero (by assumption), and which is lower semicontinuous everywhere and G\^ateaux-differentiable wherever it is finite.
		
		We apply Theorem  \ref{theo:app.pde.ekeland.variational.principle.banach} at each $\xi_R u'$, with parameters
		\begin{equation} \label{eq:app.pde.unstable.past.nodal.domain.iv}
			\varepsilon_R \triangleq \cR[\xi_R u'], \; \delta_R \triangleq \left( \int_{\Omega'} \Vert \nabla \xi_R \Vert^2 (u')^2 \, d\cL^2 \right)^{1/2} \text{.}
		\end{equation}
		We get $\psi_R \in W^{1,2}_0(\Omega')$ with
		\begin{multline} \label{eq:app.pde.unstable.past.nodal.domain.v}
			\cR[\psi_R] \leq \cR[\xi_R u'] = \varepsilon_R, \; \Vert \xi_R u' - \psi_R \Vert_{W^{1,2}(\Omega')} \leq \delta_R, \\
			\text{and} \; \Vert \delta \cR[\psi_R] \Vert_{W^{1,2}_0(\Omega')^*} \leq \frac{\varepsilon_R}{\delta_R} \text{.}
		\end{multline}
		Let us first make some observations. First, by our choice of parameters and \eqref{eq:app.pde.unstable.past.nodal.domain.i}, we have
		\begin{multline} \label{eq:app.pde.unstable.past.nodal.domain.vi}
			\varepsilon_R \Vert \xi_R u' \Vert_{L^2(\Omega')}^2	= \int_{\Omega'} \Vert \nabla \xi_R \Vert^2 (u')^2 \, d\cL^2 = \delta_R^2 \\
			\implies \Vert \xi_R u' \Vert_{L^2(\Omega')} = \varepsilon_R^{-1/2} \delta_R \text{.}
		\end{multline}
		Second, by the triangle inequality, \eqref{eq:app.pde.unstable.past.nodal.domain.v}, and \eqref{eq:app.pde.unstable.past.nodal.domain.vi}, we have
		\begin{equation} \label{eq:app.pde.unstable.past.nodal.domain.vii}
			\Vert \psi_R \Vert_{L^2(\Omega')} \leq \Vert \xi_R u' \Vert_{L^2(\Omega')} + \delta_R \leq (1+\varepsilon_R^{-1/2}) \delta_R \text{.}
		\end{equation}
		Third, by \eqref{eq:app.pde.unstable.past.nodal.domain.ii} and $u' \not \equiv 0$ a.e., the rightmost term in \eqref{eq:app.pde.unstable.past.nodal.domain.vi} satisfies
		\begin{equation} \label{eq:app.pde.unstable.past.nodal.domain.viii}
			 \lim_R \frac{\varepsilon_R}{\delta_R} = \lim_R \Vert \xi_R u' \Vert_{L^2(\Omega')}^{-2} \left( \int_{\Omega'} \Vert \nabla \xi_R \Vert^2 (u')^2 \, d\cL^2 \right)^{1/2} = 0 \text{.}	
		\end{equation}
		Recall that the G\^ateaux differential satisfies
		\begin{align*} 
			\delta \cR[\psi_R] \{ \zeta \} & = 2 \Big[ \frac{\int_{\Omega'} \left[ \langle \nabla \psi_R, \nabla \zeta \rangle + V \psi_R \zeta \right] \, d\cL^2}{\int_{\Omega'} \psi_R^2 \, d\cL^2} \\
				& \qquad \quad - \frac{\int_{\Omega'} \left[ \Vert \nabla \psi_R \Vert^2 + V \psi_R^2 \right] \, d\cL^2}{\int_{\Omega'} \psi_R^2 \, d\cL^2} \frac{\int_{\Omega'} \psi_R \zeta \, d\cL^2}{\int_{\Omega'} \psi_R^2 \, d\cL^2} \Big] \\
				& = 2 \Vert \psi_R \Vert_{L^2(\Omega')}^{-2} \cQ(\psi_R, \zeta) \\
				& \qquad \qquad - 2 \Vert \psi_R \Vert_{L^2(\Omega')}^{-2} \cR[\psi_R] \int_{\Omega'} \psi_R \zeta \, d\cL^2
		\end{align*}
		and, therefore,
		\begin{equation} \label{eq:app.pde.unstable.past.nodal.domain.ix}
			\cQ(\psi_R, \zeta) = \cR[\psi_R] \int_{\Omega'} \psi_R \zeta \, d\cL^2 + \frac{1}{2} \Vert \psi_R \Vert_{L^2(\Omega')}^2 \delta \cR[\psi_R] \{ \zeta \} \text{.}
		\end{equation}
		We plan to let $R \uparrow \infty$ while holding $\zeta$ fixed. For convenience, assume for now that $\zeta \in C^\infty_c(\Omega')$. We have, from the triangle inequality, \eqref{eq:app.pde.unstable.past.nodal.domain.v}, and \eqref{eq:app.pde.unstable.past.nodal.domain.vii}, that
		\begin{multline*}
			\Vert \psi_R \Vert_{L^2(\Omega)}^2 \Vert \delta \cR[\psi_R] \Vert_{W^{1,2}_0(\Omega')^*} \\
			\leq (1+\varepsilon_R^{-1/2})^2 \delta_R^2 \frac{\varepsilon_R}{\delta_R} = (1+\varepsilon^{1/2})^2 \delta_R \to 0 \text{ as } R \uparrow \infty {.}
		\end{multline*}
		Moreover, Cauchy-Schwarz, \eqref{eq:app.pde.unstable.past.nodal.domain.v}, and \eqref{eq:app.pde.unstable.past.nodal.domain.vii},
		\begin{multline}
			\cR[\psi_R] \int_{\Omega'} \psi_R \zeta \, d\cL^2 \leq \frac{\varepsilon_R}{\delta_R} (1+\varepsilon_R^{-1/2}) \delta_R \Vert \zeta \Vert_{L^2(\Omega')} \\
			= (\varepsilon_R + \varepsilon_R^{1/2}) \Vert \zeta \Vert_{L^2(\Omega')} \to 0 \text{ as } R \uparrow \infty \text{.}
		\end{multline}
		Plugging these last two inequalities into \eqref{eq:app.pde.unstable.past.nodal.domain.ix}, we get
		\begin{equation} \label{eq:app.pde.unstable.past.nodal.domain.x}
			\lim_{R \uparrow \infty} \cQ(\psi_R, \zeta) = 0 \text{.}
		\end{equation}
		On the other hand, since $\support \zeta \subset \subset \RR^2$, from \eqref{eq:app.pde.unstable.past.nodal.domain.v} and the definition of $\xi_R$, we see that
		\[ \lim_{R \uparrow \infty} \cQ(\psi_R, \zeta) = \lim_{R \uparrow \infty} \cQ(\xi_R u', \zeta) = \cQ(u', \zeta) \text{.} \]
		From this, together with \eqref{eq:app.pde.unstable.past.nodal.domain.x}, it follows that $\cQ(u', \zeta) = 0$ for all $\zeta \in C^\infty_c(\Omega')$. It is well-known that $C^\infty_c(\Omega')$ is dense in $W^{1,2}_0(\Omega')$, so $\cQ(u', \zeta) = 0$ for all $\zeta \in W^{1,2}_0(\Omega')$, and the claim follows.
	\end{proof}

	Let's now see how the result follows from this claim. Indeed, if $Lu' = 0$ weakly in $\Omega'$, and $u' = 0$ a.e. in $\Omega' \setminus \Omega$, then by  unique continuation we would have $u' = 0$ in $\Omega'$, so $u = 0$ a.e. in $\Omega$, so by elliptic unique continuation (again) we would have $u = 0$ in $\RR^2$---a contradiction.
\end{proof}

\section{Ekeland variational principle} \label{sec:app.pde.ekeland.variational.principle}

We mention here, without proof, a theorem of Ekeland's we make use of.

\begin{theo}[Ekeland {\cite[Theorem 2.2]{Ekeland74}}] \label{theo:app.pde.ekeland.variational.principle.banach}
	Let $F : X \to \RR \cup \{\infty\}$ be a lower semicontinuous function that is bounded from below on a Banach space $X$, and which is Fr\'echet-differentiable at all finite-valued points. Given $\varepsilon > 0$, $\delta > 0$, and a $u \in X$ such that $F(u) \leq \inf_X F + \varepsilon$, there exists $u' \in X$ such that
	\[ F(u') \leq F(u), \; \Vert u - u' \Vert \leq \delta, \; \text{ and } \Vert \delta F(u') \Vert \leq \varepsilon \delta^{-1}; \]
	here, $\delta F(u') \in X^*$ represents the Fr\'echet derivative at $u'$.
\end{theo}

\end{document}